\documentclass[sn-mathphys,Numbered]{sn-jnl}

\usepackage{graphicx}%
\usepackage{multirow}%
\usepackage{amsmath,amssymb,amsfonts}%
\usepackage{amsthm}%
\usepackage{mathrsfs}%
\usepackage[title]{appendix}%
\usepackage{xcolor}%
\usepackage{textcomp}%
\usepackage{manyfoot}%
\usepackage{booktabs}%
\usepackage{algorithm}%
\usepackage{algorithmicx}%
\usepackage{algpseudocode}%
\usepackage{listings}%
\usepackage{latexsym}
\usepackage{color}
\usepackage{bm}
\usepackage{graphicx}
\usepackage{subfigure}
\usepackage{epstopdf}
\usepackage[misc]{ifsym}

\theoremstyle{thmstyleone}%
\newtheorem{theorem}{Theorem}
\theoremstyle{thmstyletwo}%
\newtheorem{remark}{Remark}%
\newtheorem{lemma}{Lemma}
\newtheorem{assumption}{Assumption}
\newtheorem{corollary}{Corollary}

\theoremstyle{thmstylethree}%
\newtheorem{definition}{Definition}%

\begin{document}

\title[Sampling theorems associated with offset linear canonical transform by polar coordinates]{Sampling theorems associated with offset linear canonical transform by polar coordinates}


\author[1,2]{\fnm{Hui} \sur{Zhao}}\email{zhao\_hui2021@163.com}

\author*[1,2]{\fnm{Bingzhao} \sur{Li}}\email{li\_bingzhao@bit.edu.cn}

\affil[1]{\orgdiv{School of Mathematics and Statistics}, \orgname{Beijing Institute of Technology},\orgaddress{\city{ Beijing}, \postcode{100081}, \country{China}}}

\affil[2]{\orgdiv{Beijing Key Laboratory on MCAACI}, \orgname{Beijing Institute of Technology},\orgaddress{\city{ Beijing}, \postcode{100081}, \country{China}}}


\abstract{Sampling theorem for the offset linear canonical transform (OLCT) of bandlimited functions in polar coordinates is an important signal analysis tool in many fields of signal processing and medical imaging. This study investigates two sampling theorems for interpolating bandlimited and highest frequency bandlimited functions in the OLCT and offset linear canonical Hankel transform (OLCHT) domains by polar coordinates. Based on the classical Stark's interpolation formulas, we derive the sampling theorems for bandlimited functions in the OLCT and OLCHT domains, respectively. The first interpolation formula is concise and applicable. Due to the consistency of OLCHT order, the second interpolation formula is superior to the first interpolation formula in computational complexity.}




\maketitle

\section{Introduction}\label{sec1}

Fourier transform (FT) is an important mathematical tool, which is widely used in numerical computation and signal processing \cite{a,b,c,d}. Due to the need to analyze and process non-stationary signals, offset linear canonical transform (OLCT) was proposed as a generalized form of the FT \cite{1,2,3,4,5,9}. Compared to FT, the OLCT is a class of linear integral transforms with six parameters $(a, b, c, d,\tau, \eta)$. Because it adds two parameters on the basis of the linear canonical transform (LCT) \cite{6,7,8}, the OLCT has greater degrees of freedom and flexibility in applications, and is also widely used in optical systems, optical signal processing and medical imaging \cite{3,4,5,9}.

OLCT, also known as special affine Fourier transform \cite{4,28} or inhomogeneous regular transform \cite{1}, is a powerful mathematical analysis tool. Let real parameters $A=(a, b, c, d,\tau, \eta)$, $a, b, c, d,\tau, \eta\in\mathbb{R}$ satisfy $ad-bc=1$. For a function $f\in L^{2}(\mathbb{R})$, the OLCT of $f$ is defined by:
\begin{align}
	\begin{split}
		F^{A}(u)=O^{A}_{L}[f(t)](u)=
		\begin{cases}
			\int_{\mathbb{R}}f(t)h_{A}(t,u)\rm{d}t,   &b\neq0  \\
			\sqrt{d}e^{i \frac{cd}{2}(u-\tau)^{2}+i\eta u}f\left[ d(u-\tau)\right] ,    &b=0		
		\end{cases}
	\end{split}
\end{align}
where the kernel $h_{A}(t,u)$ is
\begin{align}	
	h_{A}(t,u)=K_{A}\,e^{i\left[ \frac{a}{2b}|t|^{2}+\frac{1}{b}t(\tau-u)-\frac{1}{b}u(d\tau-b\eta)+\frac{d}{2b}|u|^{2}\right] },
\end{align}
and $K_{A}=\dfrac{1}{\sqrt{i2\pi |b|}}e^{i\frac{d\tau^{2}}{2b}}$.

At present, the OLCT has achieved many achievements in its mathematical foundation, theoretical extension and practical application,  such as sampling and time-frequency analysis \cite{3,4,5,9}. Sampling is the premise and foundation of digital signal processing. Due to wider applicability, people extend the sampling theorem from the traditional FT to the OLCT.

In 2007, Adrian Stern \cite{3} obtained the sampling
theorem of bandlimited functions $f(t)\in L^{2}(\mathbb{R})$ in the OLCT domain, i.e., with $$f(t)=0,\quad  x\notin[-\frac{B_{x}}{2},\frac{B_{x}}{2}],$$  $F^{A}(u)$ can be completely recovered from its samples at points spaced $\Delta_{u} \leq\frac{2\pi}{B_{x}}|b|$, using the following interpolation formula:
\begin{align}
	\begin{split}	
		F^{A}(u)&=e^{\frac{i}{2b}\left[du^{2}-2u(d\tau-b\eta) \right] }\sum_{n\in Z}F^{A}(n\Delta_{u})\frac{\sin\left[ \pi\left( \frac{u}{\Delta_{u}}-n\right) \right] }{\pi\left( \frac{u}{\Delta_{u}}-n\right)}e^{-\frac{i}{2b}\left[d\left(n\Delta_{u} \right)^{2}-2n\Delta_{u}(d\tau-b\eta) \right] }.
		\label{r2}
	\end{split}
\end{align}
(\ref{r2}) provides the interpolation formula which permits direct reconstruction in the sampled domain.

In \cite{88}, Qiang Xiang et al. obtained the sampling
theorem for $\Omega$ bandlimited functions $f(t)\in W$ ($W$ is  subspace of the space of all integrable functions) in the OLCT domain, $f(t)$ can be exactly recovered from its samples:
\begin{align}
	\begin{split}	
		f(t)&=e^{-i\frac{a}{2b}t^{2}}\sum_{n\in Z}f(n\Delta_{t})e^{i\left[ \frac{a}{2b}(n\Delta_{t})^{2} -\frac{\tau}{b}(t-n\Delta_{t})\right] }\frac{\Delta_{t}\sin\left[\frac{\Omega}{b}(t-n\Delta_{t}) \right] }{\pi(t-n\Delta_{t})},
		\label{r3}
	\end{split}
\end{align}
where sampling interval satisfies $\Delta_{t} \leq\frac{\pi b}{\Omega}$.

In 2019, Deyun Wei, et al. derived the practical multichannel sampling expansion theorem based on new convolution structure in the OLCT domain \cite{9}.
For more sampling studies see \cite{66,77,99}. The above studies are all based on the articles of one-dimensional OLCT. 

In recent years, the sampling theorem of functions in polar coordinates has a wide range of application prospects in the fields of computed tomography (CT) \cite{11,12} and magnetic resonance imaging (MRI) \cite{13,14}. According to the existing research results, a large number of interpolation formulas for $\Omega$ bandlimited functions $f(r,\theta)$ with different bandwidth constraints from its samples have appeared in the literature \cite{15,16,17,18,19,20,21,22,23}. Due to the short application time of two-dimensional OLCT in medical imaging in polar coordinates, the theoretical system based on the OLCT is not perfect. Therefore, it makes sense to explore the sampling theorems in the OLCT and OLCHT domains in polar coordinates.

Based on our previous work \cite{24}, the purpose of this paper is to study two kinds of sampling theorems for $\Omega$ bandlimited $f(r,\theta)$ and highest frequency $\omega_{p}=\frac{K}{2\pi}$ bandlimited functions from samples at the normalized zeros $\alpha_{nj}$ in radius and at the uniformly spaced points $\frac{2\pi l}{2K+1}$ in azimuuth in the OLCT and OLCHT domains in polar coordinates. The results of the study show that due to the consistency of the order of the OLCHT, the interpolation formula in the OLCHT domain is superior to the interpolation formula in the OLCT domain in terms of computational complexity. \

The paper is organized as follows. Section \ref{2D} presents our previous research work in polar coordinates. Section \ref{2D in} gives the related results to deduce the main sampling theorems. Section \ref{OLCT} derives the sampling theorem based on $\Omega$ and $w_{p}=\frac{K}{2\pi}$ bandlimited functions $f(r,\theta)$ in the OLCT domain. Section \ref{OLCHT} derives the sampling theorem based on $\Omega$ and $w_{p}=\frac{K}{2\pi}$ bandlimited functions $f(r,\theta)$ in the OLCHT domain. Section \ref{Con} draws conclusions.\

\section{Preliminaries}
\label{2D}
In a recent work \cite{24}, we introduced the knowledge related to the OLCT and the OLCHT in polar coordinates. In order to facilitate and indepth study of the integral transformation of the OLCT, we give some mathematical definitions in polar coordinates \cite{24}.


\begin{definition}[OLCT] Let $A=(a,b;c,d)\in \mathbb{R}^{2\times2}$, $\bm{\tau}=(\tau_{1},\tau_{2})\in \mathbb{R}^{2}$ and $\bm{\eta}=(\eta_{1},\eta_{2})\in \mathbb{R}^{2}$ satisfy $\det(A)=1$. The OLCT of parameters $A$, $\bm{\tau}$, and $\bm{\eta}$ of $f(r,\theta)\in L^{2}(\mathbb{R}^{2})$ in polar coordinates is defined by 
	\begin{align}
		\begin{split}
			F^{A,\bm{\tau},\bm{\eta}}(\rho,\phi)=O^{A,\bm{\tau},\bm{\eta}}_{L}[f](\rho,\phi)=\int_{0}^{+\infty}\int_{-\pi}^{\pi}f(r,\theta)P_{A,\bm{\tau},\bm{\eta}}(r,\theta;\rho,\phi)r\rm{d}r\rm{d}\theta,
			\label{2}
		\end{split}
	\end{align}		
	where
	\begin{align}
		\begin{split}	 
			P_{A,\bm{\tau},\bm{\eta}}(r,\theta;\rho,\phi)={}&K_{A,\bm{\tau},\bm{\eta}} e^{ i\left[\frac{a}{2b}r^{2}-\frac{r\rho}{b}\cos(\theta-\phi)+\frac{d}{2b}\rho^{2}+\frac{r|\bm{\tau}|}{b}\sin(\theta+\varphi_{1})-\frac{\rho|d\bm{\tau}-b\bm{\eta}|}{b}\sin(\phi+\varphi_{2})\right] },
			\label{3}
		\end{split}
	\end{align}
	is the kernel function and 
	\begin{align}
		\begin{split}
			K_{A,\bm{\tau},\bm{\eta}}=\dfrac{1}{2\pi b}e^{i\frac{d|\bm{\tau}|^{2}}{b}},
			\label{4}
		\end{split}
	\end{align}
	where $|\bm{\tau}|^{2}=\tau_{1}^{2}+\tau_{2}^{2}$, $|\bm{\eta}|^{2}=\eta_{1}^{2}+\eta_{2}^{2}$, $\tan\varphi_{1}=\dfrac{\tau_{1}}{\tau_{2}}, \tan\varphi_{2}=\dfrac{d\tau_{1}-b\eta_{1}}{d\tau_{2}-b\eta_{2}}$, $\tau_{2}\neq0$, and $d\tau_{2}-b\eta_{2}\neq0$.	
\end{definition}
It is easy to know that if $b = 0$, the OLCT of the signal reduces to a time scaled version off multiplied by a linear chirp \cite{1}. Without loss of generality, we assume $b>0$ in the following sections.
\begin{remark}\label{re11}
	Let $A=(0,1;-1,0)$, $\bm{\tau}=\bm{0}$, and $\bm{\eta}=\bm{0}$. It follows that there is a relation between the OLCT and FT \cite{16,20,25} in polar coordinates 
	\begin{align}
		\begin{split}
			F^{A,\bm{\tau},\bm{\eta}}(\rho,\phi)=\dfrac{\ell_{1}}{b}e^{i\left[ \frac{d}{2b}\rho^{2}-\frac{\rho|d\bm{\tau}-b\bm{\eta}|}{b}\sin(\phi+\varphi_{2})\right] } F[\tilde{f}]\left( \frac{\rho}{b},\phi\right), 
			\label{8}
		\end{split} 
	\end{align}	
	where $\tilde{f}(r,\theta)=e^{i\left[ \frac{a}{2b}r^{2}+\frac{r|\bm{\tau}|}{b}\sin(\theta+\varphi_{1}) \right] }f\left(r,\theta \right)\in L^{2}(\mathbb{R}^{2})$, 
	$\varphi_{1}$ and $\varphi_{2}$ are given by (\ref{3}), and
	\begin{align}
		\begin{split}
			\ell_{1}=e^{i\frac{d|\bm{\tau}|^{2}}{b}}.
			\label{9}
		\end{split} 
	\end{align}	 
\end{remark}

\begin{definition} [Inversion formula of the OLCT] For $f(r,\theta)\in L^{2}(\mathbb{R}^{2})$ and $b>0$. The inversion formula of the OLCT with parameters $A^{-1}$, $\bm{\xi}$, and $\bm{\gamma}$ in polar coordinates takes
	\begin{align}
		\begin{split}
			f(r,\theta)=O^{A^{-1},\bm{\xi},\bm{\gamma}}_{L}[F^{A,\bm{\tau},\bm{\eta}}](r,\theta),
			\label{5}
		\end{split}
	\end{align}
	where $A^{-1}=(d,-b;-c,a)\in \mathbb{R}^{2\times2}$, $\bm{\xi}=b\bm{\eta}-d\bm{\tau}\in \mathbb{R}^{2}$, and $\,\bm{\gamma}=c\bm{\tau}-a\bm{\eta}\in \mathbb{R}^{2}$.	
\end{definition}
The definition of the OLCHT is derived from the mathematical formula of the POLCT, and they are very closely related. 
\begin{definition}[OLCHT] Let $A=(a,b;c,d)\in \mathbb{R}^{2\times2}$, $\bm{\tau}=(\tau_{1},\tau_{2})\in \mathbb{R}^{2}$, and $\bm{\eta}=(\eta_{1},\eta_{2})\in \mathbb{R}^{2}$ satisfy $\det(A)=1$. The $v$th-order OLCHT of the parameters matrix $A$, $\bm{\tau}$, and $\bm{\eta}$ of $f(r)\in L^{2}(\mathbb{R})$ in polar coordinates is defined by \cite{24} 
	\begin{align}
		\begin{split}
			H_{v}^{A,\bm{\tau},\bm{\eta}}[f](\rho)&=\frac{i^{v}\ell_{1}e^{im(\varphi_{1}-\varphi_{2})}}{b}\lambda_{1}e^{i\frac{d}{2b}\rho^{2}}\int_{0}^{+\infty}\lambda_{2}e^{i\frac{a}{2b}r^{2}}f(r)J_{v}\left(\frac{r\rho}{b}\right)r\rm{d}r,
			\label{12}
		\end{split}
	\end{align}	
	where $J_{v}$ is the $v$th-order Bessel function of the first kind and order $v\geq-\frac{1}{2}$, $\varphi_{1}$ and $\varphi_{2}$ are given by (\ref{3}), $\ell_{1}$ is given by (\ref{9}), and 
	\begin{align}
		\begin{split}
			\lambda_{1}=\sum_{m=-\infty}^{+\infty}J_{m}\left(\frac{\rho|d\bm{\tau}-b\bm{\eta}|}{b}\right),\:
			\lambda_{2}=\sum_{m=-\infty}^{+\infty}J_{m}\left(\frac{r|\bm{\tau}|}{b}\right),
			\label{13}	
		\end{split}
	\end{align}	
	where $|\bm{\tau}|^{2}=\tau_{1}^{2}+\tau_{2}^{2}$, $|\bm{\eta}|^{2}=\eta_{1}^{2}+\eta_{2}^{2}$.
\end{definition}

\begin{remark}\label{re14}
	Let $A=(0,1;-1,0)$, $\bm{\tau}=\bm{0}$, and $\bm{\eta}=\bm{0}$. we can obtain the relationship between the OLCHT and HT \cite{16,20,25}, it follows that 
	\begin{align}
		\begin{split}
			H_{v}^{A,\bm{\tau},\bm{\eta}}[f](\rho)=\frac{i^{v}\lambda_{1}\ell_{1}e^{im(\varphi_{1}-\varphi_{2})}}{b}e^{i\frac{d}{2b}\rho^{2}}H_{v}[\tilde{f}]\left( \frac{\rho}{b}\right),
			\label{17}
		\end{split}
	\end{align}	
	where $\varphi_{1}$ and $\varphi_{2}$ are given by (\ref{3}), $\ell_{1}$ is given by (\ref{9}), $\lambda_{1}$ and $\lambda_{2}$ are given by (\ref{13}), and 
	\begin{align}
		\begin{split}
			\tilde{f}(r)=\lambda_{2}e^{i\frac{a}{2b}r^{2}}f(r). 
			\label{18}	
		\end{split}
	\end{align}	
\end{remark}

\begin{definition} [Inversion formula of the OLCHT] For $f(r)\in L^{2}(\mathbb{R})$ and $b>0$. The inversion formula of $v$th-order OLCHT with parameters $A$, $\bm{\tau}$, and $\bm{\eta}$ in polar coordinates takes
	\begin{align}
		\begin{split}
			f(r)&=H_{v}^{-A^{-1},-\bm{\xi},-\bm{\gamma}}\left[ H_{v}^{A,\bm{\tau},\bm{\eta}}\left[ f\right] \right] (r)\\
			&=i^{v}\frac{\ell_{2}e^{im(\varphi_{2}-\varphi_{1})}}{b}\lambda_{2}e^{-i\frac{a}{2b}r^{2}}\int_{0}^{+\infty}\lambda_{1}e^{-i\frac{d}{2b}\rho^{2}}H_{v}^{A,\bm{\tau},\bm{\eta}}\left[ f\right](\rho)J_{v}\left(\frac{\rho r}{b}\right)\rho\rm{d}\rho,
			\label{14}
		\end{split}
	\end{align}
	where  $\lambda_{1}$ and $\lambda_{2}$ are given by (\ref{13}), and $\ell_{2}=e^{-i\frac{a|b\bm{\eta}-d\bm{\tau}|^{2}}{b}}$.
\end{definition}
\section{Some preparatory results}
\label{2D in}
Based on the above basic mathematical knowledge, we next study $\Omega$ bandlimited functions $f(r,\theta)$ and related conclusions in the OLCT and OLCHT domains.

\begin{assumption}\label{as1}
	Suppose $f(r,\theta)\in L^{2}(\mathbb{R}^{2})$ satisfy the Dirichlet condition, angularly periodic in $2\pi$, and have a Fourier series expansion 
		\begin{align}
			\begin{split}
				f(r,\theta)=\sum_{n=-\infty}^{+\infty}f_{n}\left( r\right) e^{in\theta}.
				\label{1}
			\end{split}
		\end{align}

\end{assumption}

To facilitate the proof of the sampling theorem below, we give the definitions of $\Omega_{FT}$ bandlimited functions $f(r,\theta)$ in the FT domain.  

\begin{definition} Let $f(r,\theta)$ satisfy Assumption \ref{as1}, then it is $\Omega_{FT}-$bandlimited in the FT domain to the highest frequency $\omega_{m}=\frac{K}{2\pi}$ if its Fourier expansion takes \cite{16,20}
	\begin{align}
		\begin{split}
			f(r,\theta)=\sum\limits_{n=-K}^{K}f_{n}(r)e^{in\theta}.
			\label{19}
		\end{split}
	\end{align}
\end{definition}

\begin{definition} 
	Let $f(r,\theta)$ satisfy Assumption \ref{as1} and $b>0$, then it is $\Omega-$bandlimited in the OLCT domain with the parameters $A$, $\bm{\tau}$, and $\bm{\eta}$, if $F^{A,\bm{\tau},\bm{\eta}}(\rho,\phi)=0$ for $\rho\geq \Omega$, where $F^{A,\bm{\tau},\bm{\eta}}(\rho,\phi)$ is the OLCT of $f(r,\theta)$ in polar coordinates.
\end{definition}

\begin{definition}
	Let $f(r)\in L^{2}(\mathbb{R})$ and $b>0$.\, $f(r)$ is $\Omega-$bandlimited isotropic function in the OLCHT domain with the parameters $A$, $\bm{\tau}$, and $\bm{\eta}$, if $H_{v}^{A,\bm{\tau},\bm{\eta}}[f](\rho)=0$ for $\rho\geq \Omega$, where $H_{v}^{A,\bm{\tau},\bm{\eta}}[f](\rho)$ is the $v$th-order OLCHT of $f(r)$ in polar coordinates.
\end{definition}

\begin{definition} 
	Let $f(r,\theta)$ satisfy Assumption \ref{as1} and $b>0$. Then it is $\Omega-$bandlimited in the OLCHT domain, if all of the coefficients of its Fourier series are $\Omega-$bandlimited isotropic in the OLCHT domain with the parameters $A$, $\bm{\tau}$, and $\bm{\eta}$, i.e.
	\begin{eqnarray}
		H_{v}^{A,\bm{\tau},\bm{\eta}}[f_{n}](\rho)=0 \quad \text{for} \quad \rho\geq \Omega \notag 
	\end{eqnarray} 
	where $n=0,\pm1,\pm2,\cdots.$
\end{definition} 

\begin{lemma}\label{lemma4}  
	Let $f(r,\theta)$ satisfy Assumption \ref{as1} and $b>0$.
	Then the Fourier series expansion of the OLCT of $f(r,\theta)$ has a form 
	\begin{eqnarray}
		F^{A,\bm{\tau},\bm{\eta}}(\rho,\phi)=\sum_{n=-\infty}^{+\infty}H_{2n}^{A,\bm{\tau},\bm{\eta}}\left[ f_{n}\right] \left(\rho\right) e^{in\phi}.  
		\label{35}
	\end{eqnarray}
	where $n=0,\pm1,\pm2,\cdots.$
\end{lemma}

\begin{proof}
	By (\ref{2}) and(\ref{19}), we get
	\begin{align}
		\begin{split}
			F^{A,\bm{\tau},\bm{\eta}}(\rho,\phi)&=\sum_{n=-\infty}^{+\infty}\int_{0}^{+\infty}\int_{-\pi}^{\pi}e^{in\theta}P_{A,\bm{\tau},\bm{\eta}}(r,\theta;\rho,\phi)f_{n}(r)r\rm{d}r\rm{d}\theta\\
			&=\sum_{n=-\infty}^{+\infty}\int_{0}^{+\infty}\int_{-\pi}^{\pi}e^{in\theta}\frac{\ell_{1}}{2\pi b}e^{ i\left[\frac{a}{2b}r^{2}-\frac{r\rho}{b}\cos(\theta-\phi)+\frac{d}{2b}\rho^{2}\right] }\\
			{}&\times e^{i\left[ \frac{r|\bm{\tau}|}{b}\sin(\theta+\varphi_{1})-\frac{\rho|d\bm{\tau}-b\bm{\eta}|}{b}\sin(\phi+\varphi_{2})\right] }f_{n}(r)r\rm{d}r\rm{d}\theta,
			\label{a1}
		\end{split}
	\end{align}	
	In view of the relation \cite{26}, we obtain
	\begin{align}
		\begin{split}
			e^{i \frac{r|\bm{\tau}|}{b}\sin(\theta+\varphi_{1})}=\sum_{m=-\infty}^{+\infty}J_{m}\left(\frac{r|\bm{\tau}|}{b}\right)e^{im\left( \theta+\varphi_{1}\right) },
			\label{a2}
		\end{split}
	\end{align}		
	\begin{align}
		\begin{split}
			e^{i \frac{\rho|d\bm{\tau}-b\bm{\eta}|}{b}\sin(\phi+\varphi_{2})}=\sum_{m=-\infty}^{+\infty}J_{m}\left(\frac{\rho|d\bm{\tau}-b\bm{\eta}|}{b}\right)e^{-im\left( \phi+\varphi_{2}\right) }.
			\label{a3}
		\end{split}
	\end{align}		
	It follows from a celebrated formula \cite{23}		
	\begin{align}
		\begin{split}
			J_{n}(z)=\frac{1}{2\pi}\int_{-\pi}^{+\pi}e^{i\left(n\theta-z\sin\theta \right) }\rm{d}\theta,\notag
		\end{split}
	\end{align}		
	that	
	\begin{align}
		\begin{split}
			F^{A,\bm{\tau},\bm{\eta}}(\rho,\phi)&=\sum_{n=-\infty}^{+\infty}\int_{0}^{+\infty}\int_{-\pi}^{\pi}e^{in\theta}\frac{\ell_{1}}{2\pi b}e^{ i\left[\frac{a}{2b}r^{2}-\frac{r\rho}{b}\cos(\theta-\phi)+\frac{d}{2b}\rho^{2}\right] }\lambda_{1}\lambda_{2}e^{im(\theta-\phi)}\\
			{}&\times e^{im(\varphi_{1}-\varphi_{2})}f_{n}(r)r\rm{d}r\rm{d}\theta\\
			&=\sum_{n=-\infty}^{+\infty}\frac{\ell_{1}e^{im(\theta-\phi)}}{ b}\int_{0}^{+\infty}\lambda_{1}\lambda_{2}e^{ i\left[\frac{a}{2b}r^{2}+\frac{d}{2b}\rho^{2}\right] }e^{in\phi}f_{n}(r)r{\rm{d}r}e^{-i\frac{\pi}{2}(m+n)}\\
			& \times\left\lbrace \frac{1}{2\pi}\int_{-\pi}^{\pi}e^{i\left[ (m+n)(\frac{\pi}{2}+\theta-\phi)-\frac{r\rho}{b}\sin(\frac{\pi}{2}+\theta-\phi)\right] } {\rm{d}\theta}\right\rbrace \\
			&=\sum_{n=-\infty}^{+\infty}\frac{i^{(m+n)}\ell_{1}e^{im(\theta-\phi)}}{ b}\lambda_{1}e^{i\frac{d}{2b}\rho^{2}}\int_{0}^{+\infty}\lambda_{2}e^{ i\frac{a}{2b}r^{2} }e^{in\phi}J_{m+n}(\frac{r\rho}{b})e^{in\theta}f_{n}(r)r{\rm{d}r},
			\label{a4}	
		\end{split}
	\end{align}		
	where $\lambda_{1}$ and $\lambda_{2}$ are given by (\ref{13}).\\
	From (\ref{12}), we have
	\begin{eqnarray}
		\begin{split}
			F^{A,\bm{\tau},\bm{\eta}}(\rho,\phi)=\sum_{n=-\infty}^{+\infty} H_{m+n}^{A,\bm{\tau},\bm{\eta}}[f_{n}](\rho)e^{in\phi}.
			\label{a5}
		\end{split}
	\end{eqnarray}
	Let $m=n$, we get
	\begin{eqnarray}
		\begin{split}
			F^{A,\bm{\tau},\bm{\eta}}(\rho,\phi)=\sum_{n=-\infty}^{+\infty} H_{2n}^{A,\bm{\tau},\bm{\eta}}[f_{n}](\rho)e^{in\phi}.
			\label{a6}
		\end{split}
	\end{eqnarray}
	Which completes the proof.
\end{proof}

\begin{remark}\label{re1} 
	Lemma \ref{lemma4} summarizes that the $n$th coefficient of the Fourier series of the OLCT of $f\left(r,\theta \right) $ is the $2n$th order OLCHT of the $n$th coefficient of the Fourier series of $f\left(r,\theta \right) $. If $A=(a,b;c,d)\in\mathbb{R}^{2\times2}$, $\bm{\tau}=\bm{0}$, and $\bm{\eta}=\bm{0}$, the Lemma \ref{lemma4} degenerates into the relation of the linear canonical transform (LCT) and FT \cite{23}.
\end{remark}
\begin{remark}\label{re11} 	
	When $n\to\infty$, (\ref{35}) in Lemma \ref{lemma4} can be written as
	\begin{eqnarray}
		\begin{split} 
			F^{A,\bm{\tau},\bm{\eta}}(\rho,\phi)=\sum_{n=-\infty}^{+\infty} H_{n}^{A,\bm{\tau},\bm{\eta}}[f_{n}](\rho)e^{in\phi}.
			\label{ww}
		\end{split}
	\end{eqnarray}
	Compared with Theorem 1 of \cite{24}, it is more adaptable and concise. 
\end{remark}

\begin{lemma} \label{lemma5}
	Let $f(r,\theta)$ be $\Omega-$bandlimited in the OLCT domain with parameters $A$, $\bm{\tau}$, and $\bm{\eta}$ satisfy Assumption \ref{as1} and $b>0$. Then all of the coefficients of the Fourier series of the OLCT $F^{A,\bm{\tau},\bm{\eta}}(\rho)$ are zero outside a circle of radius $\rho=\Omega$, i.e. 
	\begin{eqnarray}
		\begin{split}
			H_{n}^{A,\bm{\tau},\bm{\eta}}[f_{n}](\rho)=0, \quad \text{for} \quad \rho\geq \Omega,   
			\label{41}
		\end{split}
	\end{eqnarray}
	where $n=0,\pm1,\pm2,\cdots.$
\end{lemma}
\begin{proof}
	From (\ref{35}) in Lemma \ref{lemma4}, and $F^{A,\bm{\tau},\bm{\eta}}(\rho,\phi)$ is a periodic function of $\phi$, we can use the Parseval formula \cite{28}
	\begin{eqnarray}
		\begin{split}
			\frac{1}{2\pi}\int_{-\pi}^{\pi}|F^{A,\bm{\tau},\bm{\eta}}(\rho,\phi)|^{2}d\phi=\sum_{n=-\infty}^{+\infty} |H_{2n}^{A,\bm{\tau},\bm{\eta}}[f_{n}](\rho)|^{2}.
			\label{42}
		\end{split}
	\end{eqnarray}
	But if $F^{A,\bm{\tau},\bm{\eta}}(\rho,\phi)=0$ for $\rho\geq \Omega$, then the left-hand side of (\ref{42}) gives
	\begin{eqnarray}
		\begin{split}
			\sum_{n=-\infty}^{+\infty} |H_{2n}^{A,\bm{\tau},\bm{\eta}}[f_{n}](\rho)|^{2}=0, \quad \text{for} \quad \rho\geq \Omega.
			\label{43} 
		\end{split}
	\end{eqnarray}
	Here, (\ref{43}) implies that $$H_{n}^{A}[f_{n}](\rho)=0, \quad \text{for}  \quad \rho\geq \Omega$$. \\
	Which completes the proof.
\end{proof}

\begin{lemma} \label{lemma3}
	Let $f(r)\in L^{2}(\mathbb{R})$ be $\Omega-$bandlimited isotropic in the OLCHT domain with the parameters $A,\bm{\tau},\bm{\eta}$, and $b>0$, then the function $f(r)$ can be reconstructed at sampling point $\alpha_{vj}\in\mathbb{R}$ by
	\begin{align}
		\begin{split}
			f(r)={}&(-1)^{v}\varsigma\, e^{-i\frac{a}{2b}r^{2}}\sum\limits_{m=-\infty}^{+\infty}J_{m}\left( \frac{\mu_{1}r}{b}\right)J_{m}^{2}\left( \frac{\mu_{2}\Omega}{b}\right)\\
			{}&\times \sum\limits_{j=1}^{\infty}J_{m}\left( \frac{|\bm{\tau }|\alpha_{vj}}{b}\right)e^{i\frac{a}{2b}\alpha_{vj}^{2}}f(\alpha_{vj})\vartheta_{vj}(r),
			\label{25}
		\end{split}
	\end{align}
	where $\mu_{1}=|\bm{\tau}|$, $\mu_{2}=|d\bm{\tau}-b\bm{\eta}|$, and
	\begin{eqnarray}
		\begin{split}
			\vartheta_{vj}(r)=\frac{2b\left[\mu_{2}+\alpha_{vj}  \right]J_{v}\left( \frac{\Omega r}{b}\right) }{\Omega J_{v+1}(z_{vj})\left[\alpha_{vj}^{2}-r^{2}+2\mu_{2}\left(\alpha_{vj}-r \right)  \right] },\notag
		\end{split}
	\end{eqnarray}
	denotes the $vj$th interpolating function with the sample at $\alpha_{vj}$, $\alpha_{vj}=\frac{bz_{vj}}{\Omega}$, $z_{vj}\in\mathbb{R}$ is the $j$th zero of $J_{v}(z)$, $\varsigma=e^{i\left(\frac{d|\bm{\tau}|^{2}-a|b\bm{\eta}-d\bm{\tau}|^{2}}{b}\right) }$.	
\end{lemma}

\begin{proof}
	From (\ref{17}), due $f(r)$ is $\Omega-$bandlimited isotropic in the OLCHT domain, so $\tilde{f}(r)$ is $\frac{\Omega}{b}-$bandlimited isotropic function, such that
	\begin{eqnarray}
		H_{v}[\tilde{f}](\rho)=0 \quad \text{for} \quad \rho\geq\frac{\Omega}{b}.
		\label{26}
	\end{eqnarray}
	From (\ref{26}), $H_{v}[\tilde{f}](\rho)$ can be expanded into a Fourier-Bessel series according to \cite{16}
	\begin{align}
		\begin{split}
			H_{v}[\tilde{f}](\rho)=
			\begin{cases}
				\sum\limits_{j=1}^{\infty}\varepsilon_{j}J_{v}(\alpha_{vj}\rho),& 0<\rho<\frac{\Omega}{b}\\
				0,& \rho\geq \frac{\Omega}{b}
				\label{27}
			\end{cases}
		\end{split}
	\end{align}
	where 
	\begin{align}
		\begin{split}
			\varepsilon_{j}&=\frac{2b^{2}}{\Omega^{2}J_{v+1}^{2}(z_{vj})}\int_{0}^{\frac{\Omega}{b}}H_{v}[\tilde{f}](\rho)J_{v}(\alpha_{vj}\rho)\rho{\rm{d}}\rho=\frac{2b^{2}\tilde{f}(\alpha_{vj})}{\Omega^{2}J_{v+1}^{2}(\frac{\alpha_{vj}\Omega}{b})}.
			\label{28}		
		\end{split}
	\end{align}
	Therefore, from (\ref{17}) and (\ref{27}), we have
	\begin{align}
		\begin{split}
			H_{v}^{A,\bm{\tau},\bm{\eta}}[\tilde{f}](\rho)&{}=\begin{cases}
				\frac{i^{v}\lambda_{1}\ell_{1}e^{im(\varphi_{1}-\varphi_{2})}}{b}e^{i\frac{d}{2b}\rho^{2}}\sum\limits_{j=1}^{\infty}\varepsilon_{j}J_{v}(\frac{\alpha_{vj}\rho}{b}),& 0<\rho< \Omega\\
				0,& \rho\geq \Omega
				\label{29}	
			\end{cases}
		\end{split}
	\end{align}
	where $\ell_{1}$ is defined as (\ref{9}), $\lambda_{1}$ is given by (\ref{13}).\\
	According to (\ref{14}), the inverse $v$th-order OLCHT of (\ref{29}) enables us to write
	\begin{align}
		\begin{split}
			f(r)={}&(-1)^{v}\frac{\ell_{1}\ell_{2}}{b^{2}}e^{-i\frac{a}{2b}r^{2}}\int_{0}^{\Omega}\sum\limits_{m=-\infty}^{+\infty}J_{m}\left( \frac{\mu_{1}r}{b}\right){\rho\rm{d}\rho}\\
			{}&\times \sum\limits_{j=1}^{\infty}\varepsilon_{j}\underbrace{J_{m}^{2}\left( \frac{\mu_{2}\rho }{b}\right)J_{v}(\frac{\alpha_{vj}\rho}{b})J_{v}\left(\frac{r\rho }{b}\right)}_{\xi},
			\label{30}	
		\end{split}
	\end{align}
	where $\mu_{1}=|\bm{\tau}|$ and $\mu_{2}=|d\bm{\tau}-b\bm{\eta}|$. \\
	In view of the relation \cite{21,26}
	\begin{align}
		\begin{split}
			\xi=J_{m+v}\left( \frac{\mu_{2}\rho+\alpha_{vj}\rho}{b}\right)J_{m+v}\left( \frac{\mu_{2}\rho+r\rho }{b}\right).
			\label{31}	
		\end{split}
	\end{align}
	From the relation \cite{21,27}
	\begin{align}
		\begin{split}
			&\int_{0}^{U_{A,\bm{\tau},\bm{\eta}}}J_{m+v}\left( \frac{\mu_{2}\rho+\alpha_{vj}\rho}{b}\right)J_{m+v}\left( \frac{\mu_{2}\rho+r\rho  }{b}\right)\rho\rm{d}\rho\\
			&=\frac{b\Omega\left(\mu_{2}+\alpha_{vj} \right)}{\alpha_{vj}^{2}-r^{2}+2\mu_{2}\left( \alpha_{vj}-r\right) }J_{m+v}\left( \frac{\mu_{2}\Omega+r\Omega}{b}\right)J_{m+v+1}\left( \frac{\mu_{2}\Omega+\alpha_{vj}\Omega}{b}\right) .
		\end{split}
		\label{32}	
	\end{align}
	Using the relation (\ref{31}), we get
	\begin{align}
		\begin{split}
			&J_{m+v+1}\left( \frac{\mu_{2}\Omega+\alpha_{vj}\Omega}{b}\right)J_{m+v}\left( \Omega\frac{\mu_{2}\Omega+r\Omega}{b}\right)=J_{m}^{2}\left( \frac{\mu_{2}\Omega}{b}\right)
			J_{v+1}\left( \frac{\alpha_{vj}\Omega}{b}\right)J_{v}\left( \frac{r\Omega}{b}\right).
			\label{33}	
		\end{split}
	\end{align}
	Applying (\ref{18}), (\ref{30}), (\ref{32}), and (\ref{33}), thus
	\begin{align}
		\begin{split}
			f(r)&={}(-1)^{v}\varsigma e^{-i\frac{a}{2b}r^{2}}\sum\limits_{m=-\infty}^{+\infty}J_{m}\left( \frac{\mu_{1}r}{b}\right)J_{m}^{2}\left( \frac{\mu_{2}\Omega}{b}\right)\sum\limits_{j=1}^{\infty}J_{m}\left( \frac{|\bm{\tau }|\alpha_{vj}}{b}\right)\\
			{}&\times e^{i\frac{a}{2b}\alpha_{vj}^{2}}f(\alpha_{vj})\frac{2b\left[\mu_{2}+\alpha_{vj}  \right]J_{v}\left( \frac{r\Omega}{b}\right) }{\Omega J_{v+1}(\frac{\alpha_{vj}\Omega}{b})\left[\alpha_{vj}^{2}-r^{2}+2\mu_{2}\left(\alpha_{vj}-r \right)  \right] },
			\label{34}
		\end{split}
	\end{align}
	where $\varsigma=\ell_{1}\ell_{2}=e^{i\left(\frac{d|\bm{\tau}|^{2}-a|b\bm{\eta}-d\bm{\tau}|^{2}}{b}\right) }$. 
\end{proof}

\section{The first sampling theorem}
\label{OLCT}
For simplicity, we denote by $\mathscr{H}_{OLCT}$ the space of all functions that are $\Omega-$bandlimited in the OLCT domain and angularly bandlimited to the highest frequency $\omega_{p}=\frac{K}{2\pi}$, and by $\mathscr{H}_{OLCHT}$ the space of all functions that are $\Omega-$bandlimited in the OLHCT domain and angularly bandlimited to the highest frequency $\omega_{p}=\frac{K}{2\pi}$.\

\begin{lemma} \label{lemma6}
	Let $f(r,\theta)$ be $\Omega-$bandlimited in the OLCT domain with parameters $A$, $\bm{\tau}$, and $\bm{\eta}$ satisfying Assumption \ref{as1} and $b>0$. Then the $n$th Fourier coefficients $f_{n}\left( r\right)$ can be reconstructed at sampling point $\alpha_{nj}\in\mathbb{R}$ by 
	\begin{align}
		\begin{split}
			f_{n}(r)={}&(-1)^{n}\varsigma\, e^{-i\frac{a}{2b}r^{2}}\sum\limits_{m=-\infty}^{+\infty}J_{m}\left( \frac{\mu_{1}r}{b}\right)J_{m}^{2}\left( \frac{\mu_{2}\Omega}{b}\right)\\
			{}&\times \sum\limits_{j=1}^{\infty}J_{m}\left( \frac{|\bm{\tau }|\alpha_{nj}}{b}\right)e^{i\frac{a}{2b}\alpha_{nj}^{2}}f_{n}(\alpha_{nj})\vartheta_{nj}(r),
			\label{44}
		\end{split}
	\end{align}
	where $\alpha_{nj}=\frac{bz_{nj}}{\Omega}$, $z_{nj}\in\mathbb{R}$ is the $j$th zero of $J_{n}(z)$, and
	\begin{eqnarray}
		\begin{split}
			\vartheta_{nj}(r)=\frac{2b\left[\mu_{2}+\alpha_{nj}  \right]J_{n}\left( \frac{\Omega r}{b}\right) }{\Omega J_{n+1}(z_{nj})\left[\alpha_{nj}^{2}-r^{2}+2\mu_{2}\left(\alpha_{nj}-r \right)  \right] },\notag
		\end{split}
	\end{eqnarray} 
	here, $\varsigma$, $\mu_{1}$, and $\mu_{2}$ are the same as those stated.
\end{lemma}	
\begin{proof}
	Let $v=n$ in Lemma \ref{lemma3}, we obtain
	\begin{align}
		\begin{split}
			f(r)&={}(-1)^{n}\varsigma\, e^{-i\frac{a}{2b}r^{2}}\sum\limits_{m=-\infty}^{+\infty}J_{m}\left( \frac{\mu_{1}r}{b}\right)J_{m}^{2}\left( \frac{\mu_{2}\Omega}{b}\right)\\
			&{}\times \sum\limits_{j=1}^{\infty}J_{m}\left( \frac{|\bm{\tau }|\alpha_{nj}}{b}\right)e^{i\frac{a}{2b}\alpha_{nj}^{2}}f(\alpha_{nj})\vartheta_{nj}(r),
			\label{251}
		\end{split}
	\end{align}	
	Following from Lemma \ref{lemma5}, we can directly obtain
	\begin{align}
		\begin{split}
			f_{n}(r)&={}(-1)^{n}\varsigma\, e^{-i\frac{a}{2b}r^{2}}\sum\limits_{m=-\infty}^{+\infty}J_{m}\left( \frac{\mu_{1}r}{b}\right)J_{m}^{2}\left( \frac{\mu_{2}\Omega}{b}\right)\\
			{}&\times \sum\limits_{j=1}^{\infty}J_{m}\left( \frac{|\bm{\tau }|\alpha_{nj}}{b}\right)e^{i\frac{a}{2b}\alpha_{nj}^{2}}f_{n}(\alpha_{nj})\vartheta_{nj}(r).
			\label{441}
		\end{split}
	\end{align}
	Which completes the proof.		 
\end{proof}

The sampling theorem for the FT in polar coordinates is mentioned in \cite{16,20,25}. Let's review the classical Stark's interpolation formula \cite{15,21}. 

\begin{lemma}[Stark's interpolation formula]\label{lemma1}
	Let $f(r,\theta)$ be $\Omega_{FT}-$bandlimited in the FT domain to the highest frequency $\omega_{p}=\frac{K}{2\pi}$, satisfying Assumption \ref{as1}, and $b>0$. Then it can be uniform reconstruction at azimuthal sampling point $\left(r,\frac{2\pi l}{2K+1} \right)\in \mathbb{R}^{2}, l=0, 1, \cdots, 2K\in\mathbb{N}$ by
	\begin{eqnarray}
		\begin{split}
			f(r,\theta)=\sum_{l=0}^{2K}f\left(r,\frac{2\pi l}{2K+1} \right)o_{l}(\theta),
			\label{20}
		\end{split}
	\end{eqnarray}
	where
	\begin{align}
		\begin{split}
			o_{l}(\theta)=o\left( \theta-\frac{2\pi l}{2K+1}\right) =\frac{\sin\left[ \frac{2K+1}{2}\left( \theta-\frac{2\pi l}{2K+1}\right)\right] }{\left( 2K+1\right)\sin\left[\frac{1}{2}\left( \theta-\frac{2\pi l}{2K+1}\right) \right] },
			\label{21}
		\end{split}
	\end{align}
	denotes the $l$th interpolating function in azimuth with the sample at $\frac{2\pi l}{2K+1}$.  
\end{lemma}	

Given that the OLCT is a generalized version of the LCT, it is of great significance and value to study the sampling theorem in the field of the OLCT. The following theorem is obtained by combining Lemma \ref{lemma6} and Lemma \ref{lemma1}. 

\begin{theorem}\label{th1}
	Let $f(r,\theta)\in\mathscr{H}_{OLCT}$ satisfy Assumption \ref{as1} and $b>0$. Then it can be reconstructed at the normalized zeros $\alpha_{nj}\in\mathbb{R}$ and at the uniformly spaced points $\frac{2\pi l}{2K+1}\in\mathbb{R}$ by
	\begin{align}
		\begin{split}
			f(r,\theta)&={} \frac{(-1)^{n}\varsigma}{2K+1}e^{-i\frac{a}{2b}r^{2}}\sum\limits_{m=-\infty}^{+\infty}\sum\limits_{n=-K}^{K}\sum\limits_{j=1}^{\infty}\sum_{l=0}^{2K}e^{i\frac{a}{2b}\alpha_{nj}^{2}}f\left(\alpha_{nj},\frac{2\pi l}{2K+1} \right)\\
			{}&\times J_{m}\left( \frac{\mu_{1}r}{b}\right)J_{m}^{2}\left( \frac{\mu_{2}\Omega}{b}\right)J_{m}\left( \frac{|\bm{\tau}| \alpha_{nj}}{b}\right) \vartheta_{nj}(r)e^{in\left( \theta-\frac{2\pi l}{2K+1}\right) }.
			\label{45}
		\end{split}
	\end{align}
	where $\varsigma$, $\mu_{1}$, $\mu_{2}$, $\alpha_{nj}$, and $\vartheta_{nj}(r)$ are the same as those stated.  
\end{theorem}
\begin{proof}
	By (\ref{19}), we obtain
	\begin{eqnarray}
		\begin{split}
			f_{n}(r)=\frac{1}{2\pi}\int_{-\pi}^{\pi}f\left( r,\theta\right) e^{-in\theta}\rm{d}\theta,
			\label{46}
		\end{split}
	\end{eqnarray}
	and
	\begin{eqnarray}
		\begin{split}
			f_{n}(\alpha_{nj})=\frac{1}{2\pi}\int_{-\pi}^{\pi}f\left( \alpha_{nj},\theta\right) e^{-in\theta}\rm{d}\theta.
			\label{47}
		\end{split}
	\end{eqnarray}
	From (\ref{20}), it follows that
	\begin{eqnarray}
		\begin{split}
			f_{n}(\alpha_{nj})=\frac{1}{2\pi}\sum_{l=0}^{2K}f\left(\alpha_{nj},\frac{2\pi l}{2K+1} \right)\int_{-\pi}^{\pi}o_{l}(\theta) e^{-in\theta}{\rm{d}\theta},\:\: -K\leq n\leq K.
			\label{48} 
		\end{split}
	\end{eqnarray}  
	Following from \cite{15,21}, we obtain
	\begin{eqnarray}
		\begin{split}
			\int_{-\pi}^{\pi}o_{l}(\theta)e^{-in\theta}\rm{d}\theta=&\frac{2\pi }{2K+1}e^{-in\frac{2\pi l }{2K+1}},\:\: -K\leq n\leq K.
			\label{23}
		\end{split}
	\end{eqnarray} 
	It follows from (\ref{48}) that	
	\begin{eqnarray}
		\begin{split}
			f_{n}(\alpha_{nj})=\frac{1}{2K+1}\sum_{l=0}^{2K}f\left(\alpha_{nj},\frac{2\pi l}{2K+1} \right) e^{-in\frac{2\pi l}{2K+1}},\:\: -K\leq n\leq K.
			\label{49} 
		\end{split}
	\end{eqnarray}
	By substituting this result into (\ref{44}), we get
	\begin{align}
		\begin{split}
			f_{n}(r)&={} \frac{(-1)^{n}\varsigma}{2K+1}e^{-i\frac{a}{2b}r^{2}}\sum\limits_{m=-\infty}^{+\infty}J_{m}\left( \frac{\mu_{1}r}{b}\right)J_{m}^{2}\left( \frac{\mu_{2}\Omega}{b}\right)\sum\limits_{j=1}^{\infty}J_{m}\left( \frac{|\bm{\tau }|\alpha_{nj}}{b}\right)\\
			{}&\times e^{i\frac{a}{2b}\alpha_{nj}^{2}}\vartheta_{nj}(r)\sum_{l=0}^{2K}f\left(\alpha_{nj},\frac{2\pi l}{2K+1} \right) e^{-in\frac{2\pi l}{2K+1}},
			\label{50}
		\end{split}
	\end{align}
	for all $-K\leq n\leq K$. \\ 
	Hence
	\begin{align}
		\begin{split}
			f(r,\theta)&={}\sum\limits_{n=-K}^{K}f_{n}(r)e^{in\theta}\frac{(-1)^{n}\varsigma}{2K+1}e^{-i\frac{a}{2b}r^{2}}\sum\limits_{m=-\infty}^{+\infty}J_{m}\left( \frac{\mu_{1}r}{b}\right)J_{m}^{2}\left( \frac{\mu_{2}\Omega}{b}\right)\\
			{}&\times \sum\limits_{n=-K}^{K}\sum\limits_{j=1}^{\infty}J_{m}\left( \frac{|\bm{\tau }|\alpha_{nj}}{b}\right)e^{i\frac{a}{2b}\alpha_{nj}^{2}}\vartheta_{nj}(r)\sum_{l=0}^{2K}f\left(\alpha_{nj},\frac{2\pi l}{2K+1} \right) e^{in\left( \theta-\frac{2\pi l}{2K+1}\right) }.
			\label{51}
		\end{split}
	\end{align}
	Which completes the proof.
\end{proof}

\begin{remark} \label{re3}   	
	According to (\ref{45}) in Theorem \ref{th1}, it is clear that the required number of samples goes as $$\left[ (2K+1)N\right]^{2}, \quad N\to\infty$$ where the number of normalized zeros takes $(2K+1)N(N\to\infty)$.
\end{remark}

\begin{remark} \label{re4}  	
	The (\ref{45}) in Theorem \ref{th1} is the classical interpolation formula \cite{16,20} of the FT domain angular periodic function when $A=\left(0,1;-1,0 \right)$, $\bm{\tau}=\bm{0}$, and $\bm{\eta}=\bm{0}$, and its is the interpolation formula \cite{21,22} of the LCT domain when $A=\left(a,b;c,d\right)\in\mathbb{R}^{2\times2}$, $\bm{\tau}=\bm{0}$, and $\bm{\eta}=\bm{0}$. However, if the function is nonbandlimited in the FT or LCT domain, the classical results can be misleading. Fortunately, it is bandlimited in the OLCT domain. Therefore, the interpolation formula derived in (\ref{45}) can solve the traditional nonbandlimited functions processing problems in the FT or LCT domain.
\end{remark}

\begin{corollary} \label{corollary1}
	Let $f(r,\theta)\in\mathscr{H}_{OLCT}$ satisfy Assumption \ref{as1} and $b>0$. Then the OLCT $F^{A,\bm{\tau},\bm{\eta}}(\rho,\phi)$ of $f(r,\theta)$ can be reconstructed at the normalized zeros $\alpha_{nj}\in\mathbb{R}$ and at the uniformly spaced points $\frac{2\pi l}{2K+1}\in\mathbb{R}$ by
	\begin{align}
		\begin{split}
			F^{A,\bm{\tau},\bm{\eta}}(\rho,\phi)&={} \frac{(-1)^{n}\varsigma}{2K+1}e^{i\frac{d}{2b}\rho^{2}}\sum\limits_{m=-\infty}^{+\infty}\sum\limits_{n=-K}^{K}\sum\limits_{j=1}^{\infty}\sum_{l=0}^{2K}e^{-i\frac{d}{2b}\alpha_{nj}^{2}}J_{m}\left( \frac{\mu_{2}\rho}{b}\right)J_{m}^{2}\left( \frac{\mu_{1}\Omega}{b}\right)\\
			{}&\times J_{m}\left( \frac{|\bm{\tau}| \alpha_{nj}}{b}\right)  F^{A,\bm{\tau},\bm{\eta}}\left(\alpha_{nj},\frac{2\pi l}{2K+1} \right) \vartheta_{nj}(\rho)e^{in\left( \phi-\frac{2\pi l}{2K+1}\right) }.
			\label{45}
		\end{split}
	\end{align}
	where $\varsigma$, $\mu_{1}$, $\mu_{2}$, $\alpha_{nj}$, and $\vartheta_{nj}(\rho)$ are the same as those stated.  
\end{corollary}	

\begin{proof}
	Because of the inversion formula of the OLCT, we obtain
	\begin{eqnarray} 
		O^{A^{-1},\bm{\xi},\bm{\gamma}}_{L}[F^{A,\bm{\tau},\bm{\eta}}](r,\theta)=0 \quad \text{for} \quad \rho\geq \Omega, \notag 
	\end{eqnarray} 
	it implies that  $F^{A,\bm{\tau},\bm{\eta}}\in\mathscr{H}_{OLCT}$ with $A^{-1}$, $\bm{\xi}$, and $\bm{\gamma}$.\\
	Following from Remark \ref{re11}, we can directly obtain	
	\begin{eqnarray}
		\begin{split} 
			F^{A,\bm{\tau},\bm{\eta}}(\rho,\phi)=\sum_{n=-\infty}^{+\infty} H_{n}^{A,\bm{\tau},\bm{\eta}}[f_{n}](\rho)e^{in\phi}.
			\label{ww1}
		\end{split}
	\end{eqnarray}
	According to (\ref{ww1}), it implies that $F^{A,\bm{\tau},\bm{\eta}}$ satisfies Assumption  \ref{as1}, and
	\begin{eqnarray}
		\begin{split} 
			F^{A,\bm{\tau},\bm{\eta}}(\rho,\phi)=\sum_{n=-K}^{K} H_{-n}^{A,\bm{\tau},\bm{\eta}}[f_{-n}](\rho)e^{in\phi}.
			\label{ww2}
		\end{split}
	\end{eqnarray}
	By using Theorem \ref{th1}, we get
	\begin{align}
		\begin{split}
			F^{A,\bm{\tau},\bm{\eta}}(\rho,\phi)&={} \frac{(-1)^{n}\varsigma}{2K+1}e^{i\frac{d}{2b}\rho^{2}}\sum\limits_{m=-\infty}^{+\infty}\sum\limits_{n=-K}^{K}\sum\limits_{j=1}^{\infty}\sum_{l=0}^{2K}e^{-i\frac{d}{2b}\alpha_{nj}^{2}}J_{m}\left( \frac{\mu_{2}\rho}{b}\right)J_{m}^{2}\left( \frac{\mu_{1}\Omega}{b}\right)\\
			{}&\times J_{m}\left( \frac{|\bm{\tau}| \alpha_{nj}}{b}\right)F^{A,\bm{\tau},\bm{\eta}}\left(\alpha_{nj},\frac{2\pi l}{2K+1} \right) \vartheta_{nj}(\rho)e^{in\left( \phi-\frac{2\pi l}{2K+1}\right) }.
		\end{split}
	\end{align}
	Which completes the proof.
\end{proof}
\section{The second sampling theorem}
\label{OLCHT}
Inspired by the classical interpolation formula \cite{15,21}, this section mainly studies the sampling theorem for $f(r,\theta)\in\mathscr{H}_{OLCHT}$ from samples at the normalized zeros $\alpha_{nj}$ in radius and at the uniformly spaced points $\frac{2\pi l}{2K+1}$ in azimuuth in the OLCHT domain in polar coordinates.

\begin{lemma} \label{lemma7}
	Let $f(r,\theta)$ be $\Omega-$bandlimited in the OLCHT domain with parameters $A$, $\bm{\tau}$, and $\bm{\eta}$ satisfying Assumption \ref{as1} and $b>0$. Then the $n$th Fourier coefficients $f_{n}\left( r\right)$ can be reconstructed at sampling point $\alpha_{vj}\in\mathbb{R}$ by 
	\begin{align}
		\begin{split}
			f_{n}(r)&={}(-1)^{v}\varsigma\, e^{-i\frac{a}{2b}r^{2}}\sum\limits_{m=-\infty}^{+\infty}J_{m}\left( \frac{\mu_{1}r}{b}\right)J_{m}^{2}\left( \frac{\mu_{2}\Omega}{b}\right)\\
			{}&\times \sum\limits_{j=1}^{\infty}J_{m}\left( \frac{|\bm{\tau }|\alpha_{vj}}{b}\right)e^{i\frac{a}{2b}\alpha_{vj}^{2}}f_{n}(\alpha_{vj})\vartheta_{vj}(r),
			\label{52}
		\end{split}
	\end{align}
	where $\varsigma, \mu_{1}, \mu_{2}, \alpha_{vj}$, and $\vartheta_{nj}(r)$ are the same as those stated.  
\end{lemma}	

\begin{proof}
	Replacing $f(r)$ in Lemma \ref{lemma3} with $f_{n}(r)$, we can directly obtain
	\begin{align}
		\begin{split}
			f_{n}(r)&={}(-1)^{v}\varsigma\, e^{-i\frac{a}{2b}r^{2}}\sum\limits_{m=-\infty}^{+\infty}J_{m}\left( \frac{\mu_{1}r}{b}\right)J_{m}^{2}\left( \frac{\mu_{2}\Omega}{b}\right)\\
			{}&\times \sum\limits_{j=1}^{\infty}J_{m}\left( \frac{|\bm{\tau }|\alpha_{vj}}{b}\right)e^{i\frac{a}{2b}\alpha_{vj}^{2}}f_{n}(\alpha_{vj})\vartheta_{vj}(r).
			\label{25}
		\end{split}
	\end{align}
	Which completes the proof.
\end{proof}

According to Lemma \ref{lemma1} and \ref{lemma7}, An interpolation formula is obtained in the OLCHT domain. This interpolation formula is essentially different from Theorem \ref{th1} due to the consistency of the OLCHT order, where the sampling points are normalized zeros of the Bessel function on radius. Theorem \ref{th2} better reduces the number of normalized zeros.

\begin{theorem}\label{th2}
	Let $f(r,\theta)\in\mathscr{H}_{OLCHT}$ satisfy Assumption \ref{as1} and $b>0$. Then it can be reconstructed at the normalized zeros $\alpha_{vj}\in\mathbb{R}$ and at the uniformly spaced points $\frac{2\pi l}{2K+1}\in\mathbb{R}$ by
	\begin{align}
		\begin{split}
			f(r,\theta)&={}(-1)^{v}\varsigma e^{-i\frac{a}{2b}r^{2}}\sum\limits_{j=1}^{\infty}\sum_{l=0}^{2K}\sum\limits_{m=-\infty}^{+\infty}e^{i\frac{a}{2b}\alpha_{vj}^{2}}J_{m}\left( \frac{\mu_{1}r}{b}\right)J_{m}^{2}\left( \frac{\mu_{2}\Omega}{b}\right)\\
			{}&\times J_{m}\left( \frac{|\bm{\tau }| \alpha_{vj}}{b}\right) \vartheta_{vj}(r)f\left(\alpha_{vj},\frac{2\pi l}{2K+1} \right)o_{l}(\theta).
			\label{53}
		\end{split}
	\end{align}
	where $\varsigma, \mu_{1}, \mu_{2}, \alpha_{vj}, \vartheta_{vj}(r)$, and $o_{l}(\theta)$ are the same as those stated.  
\end{theorem}
\begin{proof}
	Replacing $\alpha_{nj}$ in (\ref{49}) with $\alpha_{vj}$, we have
	\begin{eqnarray}
		\begin{split}
			f_{n}(\alpha_{vj})=\frac{1}{2K+1}\sum_{l=0}^{2K}f\left(\alpha_{vj},\frac{2\pi l}{2K+1} \right) e^{-in\frac{2\pi l}{2K+1}},\:\: -K\leq n\leq K.
			\label{54}  
		\end{split}
	\end{eqnarray}
	Using the (\ref{52}), we get
	\begin{align}
		\begin{split}
			f_{n}(r)&={} \frac{(-1)^{v}\varsigma}{2K+1}e^{-i\frac{a}{2b}r^{2}}\sum\limits_{m=-\infty}^{+\infty}J_{m}\left( \frac{\mu_{1}r}{b}\right)J_{m}^{2}\left( \frac{\mu_{2}\Omega}{b}\right)\sum\limits_{j=1}^{\infty}J_{m}\left( \frac{|\bm{\tau }|\alpha_{vj}}{b}\right)\\
			{}&\times e^{i\frac{a}{2b}\alpha_{vj}^{2}}\vartheta_{vj}(r)\sum_{l=0}^{2K}f\left(\alpha_{vj},\frac{2\pi l}{2K+1} \right) e^{-in\frac{2\pi l}{2K+1}},
			\label{55}
		\end{split}
	\end{align}
	for all $-K\leq n\leq K$. \\ 
	By the following triangle sum formula \cite{21}
	\begin{align}
		\begin{split}
			\left(2K+1 \right) o_{l}(\theta)=\sum\limits_{n=-K}^{K}e^{in\left( \theta-\frac{2\pi l}{2K+1}\right) }.
			\label{56}
		\end{split}
	\end{align}
	Applying (\ref{19}) and (\ref{56}), we obtain
	\begin{align}
		\begin{split}
			f(r,\theta)={}&\sum\limits_{n=-K}^{K}f_{n}(r)e^{in\theta}\\
			={}&(-1)^{v}\varsigma e^{-i\frac{a}{2b}r^{2}}\sum\limits_{m=-\infty}^{+\infty}J_{m}\left( \frac{\mu_{1}r}{b}\right)J_{m}^{2}\left( \frac{\mu_{2}\Omega}{b}\right)e^{i\frac{a}{2b}\alpha_{vj}^{2}}\\
			{}\times& \sum\limits_{j=1}^{\infty}\sum_{l=0}^{2K}J_{m}\left( \frac{|\bm{\tau }| \alpha_{vj}}{b}\right)\vartheta_{vj}(r)f\left(\alpha_{vj},\frac{2\pi l}{2K+1} \right) o_{l}(\theta).
			\label{57}
		\end{split}
	\end{align}
	Which completes the proof.
\end{proof}

\begin{remark} \label{re5}   	
	According to (\ref{53}) in Theorem \ref{th2}, it is clear that the required number of samples goes as $$(2K+1) N^{2}, \quad N\to\infty$$ where the number of normalized zeros takes $N(N\to\infty)$.
\end{remark}

\begin{remark} \label{re6}  	
	The (\ref{53}) in Theorem \ref{th2} is the classical interpolation formula \cite{16,20} of the HT domain  when $A=\left(0,1;-1,0 \right)$, $\bm{\tau}=\bm{0}$, and $\bm{\eta}=\bm{0}$, and its is the interpolation formula \cite{21,22} of the LCHT domain when $A=\left(a,b;c,d\right)\in\mathbb{R}^{2\times2}$, $\bm{\tau}=\bm{0}$, and $\bm{\eta}=\bm{0}$. So, the interpolation formula derived in (\ref{53}) can solve the nonbandlimited functions processing problem in the Hankel transform (HT) or the linear canonical Hankel transform (LCHT) domain.
\end{remark}

\begin{remark} \label{re7}  	
	It is emphasized here that the interpolation formula(\ref{53}) is essentially different from (\ref{45}) in Theorem \ref{th1}, because the transform domain in which the reconstructed object is located is different. By comparing (\ref{45}) and (\ref{53}), it is obvious that the second interpolation formula is better than the first interpolation formula in terms of computational complexity.
\end{remark}

\begin{corollary} \label{corollary2}
	Let $f(r,\theta)\in\mathscr{H}_{OLHCT}$ satisfy Assumption \ref{as1} and $b>0$. Then the OLCT $F^{A,\bm{\tau},\bm{\eta}}(\rho,\phi)$ of $f(r,\theta)$ can be reconstructed at the normalized zeros $\alpha_{nj}\in\mathbb{R}$ and at the uniformly spaced points $\frac{2\pi l}{2K+1}\in\mathbb{R}$ by	
	\begin{align}
		\begin{split}
			F^{A,\bm{\tau},\bm{\eta}}(\rho,\phi)&={}(-1)^{v}\varsigma e^{i\frac{d}{2b}\rho^{2}}\sum\limits_{j=1}^{\infty}\sum_{l=0}^{2K}\sum\limits_{m=-\infty}^{+\infty}e^{-i\frac{a}{2b}\alpha_{vj}^{2}}J_{m}\left( \frac{\mu_{2}\rho}{b}\right)J_{m}^{2}\left( \frac{\mu_{1}\Omega}{b}\right)\\
			{}&\times J_{m}\left( \frac{|\bm{\tau }| \alpha_{vj}}{b}\right) \vartheta_{vj}(\rho)F^{A,\bm{\tau},\bm{\eta}}\left(\alpha_{vj},\frac{2\pi l}{2K+1} \right)o_{l}(\phi).
			\label{533}
		\end{split}
	\end{align}
	where $\varsigma, \mu_{1}, \mu_{2}, \alpha_{vj}, \vartheta_{vj}(\rho)$, and $o_{l}(\phi)$ are the same as those stated.  
\end{corollary}
\begin{proof}
	According to (\ref{ww1}), it implies that $F^{A,\bm{\tau},\bm{\eta}}$ satisfies Assumption  \ref{as1}, and
	\begin{eqnarray}
		\begin{split} 
			F^{A,\bm{\tau},\bm{\eta}}(\rho,\phi)=\sum_{n=-K}^{K} H_{-n}^{A,\bm{\tau},\bm{\eta}}[f_{-n}](\rho)e^{in\phi}.
			\label{ww2}
		\end{split}
	\end{eqnarray}
	By using Theorem \ref{th2}, we have
	\begin{align}
		\begin{split}
			F^{A,\bm{\tau},\bm{\eta}}(\rho,\phi)&={}(-1)^{v}\varsigma e^{i\frac{d}{2b}\rho^{2}}\sum\limits_{j=1}^{\infty}\sum_{l=0}^{2K}\sum\limits_{m=-\infty}^{+\infty}e^{-i\frac{a}{2b}\alpha_{vj}^{2}}J_{m}\left( \frac{\mu_{2}\rho}{b}\right)J_{m}^{2}\left( \frac{\mu_{1}\Omega}{b}\right)\\
			{}&\times J_{m}\left( \frac{|\bm{\tau }| \alpha_{vj}}{b}\right) \vartheta_{vj}(\rho)F^{A,\bm{\tau},\bm{\eta}}\left(\alpha_{vj},\frac{2\pi l}{2K+1} \right)o_{l}(\phi).
		\end{split}
	\end{align}
	Which completes the proof.
\end{proof}

\section{Conclusions}
\label{Con}
This paper studies the sampling theorems of bandlimited functions in the OLCT and OLCHT domains in polar coordinates, that is, interpolating uniform samples in radius and interpolating the highest frequency range samples in azimuth, where the sampling points are normalized zeros of the Bessel function on radius. The first interpolation formula is a generalization of the FT and LCT domains, which is more general. The second interpolation formula is superior to the first interpolation formula in terms of computational complexity due to the consistency of the OLCHT order. \


\section*{Declarations}
The authors declare that they have no known competing financial interests or personal relationships that could have appeared to influence the work reported in this paper.

\section*{Data Availability Statement}
Not applicable.

\section*{Code Availability Statement}
Not applicable.

\bibliography{sn-bibliography}

\end{document}